\documentclass[reqno, 11pt]{amsart}
\usepackage{amsmath, amsthm, amsfonts, amssymb}
\usepackage{color,xcolor}

\newtheorem{theorem}{Theorem}[section]
\newtheorem{lemma}[theorem]{Lemma}
\newtheorem{question}[theorem]{Question}

\theoremstyle{definition}
\newtheorem{definition}[theorem]{Definition}

\newtheorem{example}[theorem]{Example}

\newtheorem{remark}[theorem]{Remark}
\numberwithin{equation}{section}

 \numberwithin{equation}{section}

\newcommand{\C}{\mathbb{C}}
\newcommand{\D}{\mathbb{D}}

\newcommand{\F}{\mathcal{F}}

\newcommand{\Cn}{\mathbb{C}^{n}}

 \allowdisplaybreaks
 
\begin{document}

\author[Pratiksha]{PRATIKSHA}
\address{
\begin{tabular}{lll}
& Pratiksha \\
&Department of Mathematics\\
&University of Jammu\\
&Jammu-180006\\ 
&India\\
\end{tabular}}
\email{pratikshapakhetra1999@gmail.com}

\title[On normality and $\varphi$-normality...]{On normality and $\varphi$-normality of holomorphic functions in several complex variables}

\begin{abstract}

In this paper, we investigate $\varphi$-normal functions and normal families of holomorphic functions concerning total derivatives in $\mathbb{C}^{n}.$ More precisely, we prove a sufficient condition for a holomorphic function defined on an open unit ball in $\Cn$ satisfying certain conditions involving higher order partial derivatives to be $\varphi$-normal. Furthermore, by using differential inequalities involving total differential polynomials in $\Cn,$ we establish some normality criteria for holomorphic functions in $\Cn$ which generalize some known results.

\end{abstract}

\maketitle

\section{Introduction and main results}
\renewcommand{\thefootnote}{\fnsymbol{footnote}}
\footnotetext{2020 {\it Mathematics Subject Classification}. Primary 32A18, 32A19; Secondary 32A22.}
\footnotetext{{\it Keywords and phrases}. Normal families, normal functions, total derivatives, holomorphic functions.}
Let $\Omega$ be a domain in $\C.$ A family of  meromorphic functions defined
in $\Omega \subset \C$ is said to be \emph{normal in $\Omega$} if every sequence in the
family admits a subsequence that converges spherically uniformly in compact subsets of $\Omega$ either to a
meromorphic function or to a function identically equal to infinity. The theory of normal families occupies a central position in complex analysis. Its origins can be traced back to the pioneering work of Montel \cite{MON} in the early twentieth century. Since then, this theory has developed into a powerful analytical framework with far reaching applications in complex analysis, dynamics of rational and transcendental maps, bicomplex analysis, and several other branches of modern mathematics (see, for example, \cite{ARB, NIK-1, CHA, RAH3, PAN}). Recently, methods of normal families have seen a lot a applications in harmonic mappings (see \cite{NIK, BOH, DENG}) which, in turn, have been employed to study fluid flow problems (see \cite{AL}).

\medskip

A simultaneous development in this theory is the concept of \emph{normal functions}, introduced implicitly by Yosida \cite{YOS} and Noshiro \cite{NOS} and later formalized by Lehto and Virtanen \cite{LEH}. A meromorphic function defined on the open unit disk $\D$ is called normal if its post composition with conformal automorphisms of the unit disk forms a normal family. This notion admits an equivalent characterization in terms of the boundedness of the spherical derivative as
\begin{definition}
A meromorphic function $f$ defined in the open unit disk $\mathbb{D}\subset \C$ is called normal in $\mathbb{D}$ if
 \begin{equation}\label{eq:n1}
     \sup_{\zeta \in \mathbb{D}}(1-|\zeta|^2)f^{\#}(\zeta) < \infty,
 \end{equation}
 where $f^{\#}(\zeta)$ denotes the spherical derivative of $f(\zeta).$
 \end{definition}
The condition \eqref{eq:n1} above means that $f$ is Lipschitz when regarded as a function from the hyperbolic disk $\D$ into the extended complex plane endowed with the chordal distance (see \cite{LEH}).
 Motivated by the desire to enlarge the class of normal functions, Aulaskari and
R\"atty\"a \cite{AUL}, introduced the notion of $\varphi$-normal functions, replacing
the classical growth condition involving the hyperbolic metric with a more
flexible scale governed by a smoothly increasing function $\varphi$. This
extension has since attracted considerable attention and has led to a wealth of new results, especially concerning refined normality criteria and boundary
behaviour. For more detailed view of theory of normal families and normal functions in $\C$, the reader is advised to go through \cite{SCH}.\\

Despite the maturity of the theory in one complex variable, the extension of these ideas to the setting of several complex variables presents substantial
challenges. The intricate geometry of domains in $\C^n$, renders many classical techniques ineffective or incomplete. Consequently, the development
of theory of normal families in $\C^n$ has become an active and vibrant area of research. The present paper, therefore, is devoted to provide more richness into the theory of normal families in $\Cn.$\\

Let $U$ be a domain in $\C^n$. Throughout this paper, by a \emph{holomorphic function} in $U$, we mean a $\C$-valued holomorphic function defined in $U$.
Analogous to the one-variable case, a family of holomorphic functions defined in $U \subset \C^n$ is said to be \emph{normal} in $U$ if every sequence in the family admits a subsequence that converges locally uniformly in $U$ either to a holomorphic function or to the function identically equal to infinity. To facilitate our results, we shall be using the following notations:
\begin{align*} 
 z &= (z_{1},~z_{2},~\ldots,~z_{n}) \in \Cn, z_{j} \in \C,~~~~~~~ j= 1, 2,\ldots, n,\\
0 &= (0,~0,~\ldots,~0) \in \Cn,\\
\|z-w\| &= \sqrt{\sum\limits_{j=1}^{n}|z_{j}-w_{j}|^{2},}\\
\Delta(w,\delta)&=\{z\in \Cn: \|z-w\|<\delta\},~ \delta>0,\\
\mathbb{B} &= \{z\in \Cn: \|z\|<1\}.
\end{align*}\\
A fundamental tool in the study of normality is the notion of the spherical derivative, whose extension to several complex variables was formulated via the complex Hessian as;
Let $ \omega = (\omega_1,~ \omega_2,~\ldots, ~\omega_n) \in \Cn$, and for every function $h \in \mathcal{C}^2(U)$, at each $z \in U$, define a Hermitian form
\begin{equation}\label{1.4}
L_{z}(h, \omega) = \sum\limits_{j,~ k=1}^{n}\frac{\partial^2h(z)}{\partial z_j\partial\overline{z}_k}\omega_j\overline{\omega}_k,
\end{equation}
called the Hessian of $h$ at $z$. For a holomorphic function $h$ in $U$, the spherical derivative is defined by
\begin{equation}\label{1.5}
h^{\#}(z) = \sup_{\|\omega\|=1}\sqrt{L_{z}(log(1+|h|^2), \omega)}.
\end{equation}
 This definition coincides with the classical spherical derivative when $n=1$ and admits the equivalent representation
$$h^{\#}(z) = \frac{|h'(z)|}{1+|h(z)|^2}.$$
as shown in \cite{DOV}.
Further, by straightforward calculations, one can obtain
\begin{equation}\label{spherical2}
h^{\#}(z)=\sup_{\|\omega\|=1}
\frac{|\langle\nabla h(z),\overline{\omega}\rangle|}
{1+|h(z)|^2},
\end{equation}
 where $\nabla h(z)$ denotes the complex gradient of $h$, and $\langle z,w \rangle$,  with $z, w \in \Cn$ is the standard inner product in $\Cn$.\\\\
Employing this notion of spherical derivative, Zhu et al. \cite{JCA} extended the concept of smoothly increasing functions and the notion of $\varphi$-normality to $\C^n$ as 
\begin{definition}\label{1.6}
A function $\varphi : [0,1) \longrightarrow (0,\infty)$ is said to be \emph{smoothly increasing} if
\begin{equation}
\varphi(r)(1-r) \longrightarrow \infty \quad \text{as } r \longrightarrow 1^{-}, 
\end{equation}
and
\begin{equation}
\mathcal{R}_a(w):=
\frac{\varphi\!\left(\left\|a+{w}/{\varphi(\|a\|)}\right\|\right)}{\varphi(\|a\|)}
\longrightarrow 1 \quad \text{as }~ \|a\| \longrightarrow 1^{-},
\end{equation}
uniformly in compact subsets of $\mathbb{C}^n$.  
\end{definition}
Given such a function $\varphi$, a holomorphic function $h$ in $\mathbb{B}$ is said to be
\emph{$\varphi$-normal} if
\begin{equation}
\sup_{w\in\mathbb{B}}\frac{ h^{\#}(w)}{\varphi(\|w\|)} < \infty, 
\end{equation}
where $h^{\#}(w)$ is given by \eqref{spherical2}.\\\\
Also, recall that for multi-index $J = (j_1,~ j_2,~\ldots,~j_n) \in \mathbb{N}^n,~ |J| = \sum\limits_{k=1}^{n}j_k,$ at each~ $w \in \Cn,$ the partial derivative operator , $D^J$
is given by
\begin{equation}\label{1.7}
    D^J = \frac{\partial^{|J|}}{\partial w_1^{j_1}\partial w_2^{j_2}\ldots\partial w_n^{j_n}}.
\end{equation}\\
Using Definition \ref{1.6} and the operator defined in \eqref{1.7}, Zhu et al.~\cite{JCA} obtained sufficient conditions for $\varphi$-normality involving higher-order partial derivatives. More precisely, they proved the following result:
\begin{theorem}\label{1.08}
  Let $k$ be a positive integer and $f$
be a holomorphic function in $\mathbb{B}\subset \Cn$ such that
\begin{equation*}
\sup\{ |D^{I}f(z)|: f(z) = 0,~ I \in \mathbb{N}^n,~ |I| = 1,~2,~\ldots,~ k-1\} < \infty.
\end{equation*}
If there exists a set $E$ consisting of three distinct finite complex numbers such that
\begin{equation*}
\sup_{z\in f^{-1}(E)}\frac{1}{\varphi(\|z\|)^{k}}\frac{|D^{J}f(z)|}{1 +|f(z)|^{k+1}} < \infty,~ J \in \mathbb{N}^{n},~ |J|= k,
\end{equation*}
then $f$ is $\varphi$-normal.
\end{theorem}
\medskip

\noindent This result naturally raises the question of whether the cardinality of the set $E$ is optimal. In this paper, we address this question affirmatively by demonstrating that the set $E$ can indeed be reduced from three points to two, provided additional control is imposed on higher-order derivatives. This constitutes our first main result and significantly sharpens Theorem \ref{1.08} as
\begin{theorem}\label{T1}
Let $k$ be a positive integer and $f$
be a holomorphic function in $\mathbb{B}\subset \Cn$ such that
\begin{equation}\label{2.1}
\sup\{ |D^{I}f(z)|: f(z) = 0,~ I \in \mathbb{N}^n,~ |I| = 1,~2,~\ldots,~ k-1\} < \infty.
\end{equation}
If there exists a set $E$ consisting of two distinct finite complex numbers such that
\begin{equation}\label{2.2}
\sup_{z\in f^{-1}(E)}\frac{1}{\varphi(\|z\|)^{k}}\frac{|D^{J}f(z)|}{1 +|f(z)|^{k+1}} < \infty,~ J \in \mathbb{N}^{n},~ |J|= k,
\end{equation}
and
\begin{equation}\label{2.3}
\sup_{z\in f^{-1}(E)}\frac{|D^{L}f(z)|}{1 +|D^{J}f(z)|^\frac{k+1}{k}} < \infty,~ L \in \mathbb{N}^{n}, ~|L|= k+1,
\end{equation}
then $f$ is $\varphi$-normal.
\end{theorem}
\noindent Theorem \ref{T1} for plannar harmonic mappings is due to Bharti and Thin \cite{NIK}. The next objective of this paper is to study the theory of normal families involving total derivatives in several complex variables, introduced by Lu \cite{JIN} as
\begin{definition}
Let $z = (z_{1},~z_{2},~\ldots,~z_{n}) \in \Cn$, and $f$ be an entire function in $\Cn.$ The total derivative $D_*f$ of $f$ is defined by $D_*f(z) := 
\sum\limits_{j=1}^{n}z_{j}f_{z_{j}}(z),$ where $f_{z_j}$ is the partial derivative of $f$ with respect to $z_{j},~ j =1,~2,~\ldots,~n.$ The l-th order total derivative $D_*^{l}f(z)$ is defined recursively by
$$D_*^{l}f(z):= D_*(D_*^{l-1}f)(z),~ l\geq 1 ~\mbox{with}~ D_*^{0}f(z) := f(z).$$
    
\end{definition}

\noindent
\begin{remark}
It is worth mentioning that the total derivative of a transcendental entire function is again a transcendental entire function. In contrast, this property need not hold for partial derivatives.
\end{remark}
Building upon this idea, Cao and Liu \cite{LIU} established analogues of Marty’s theorem, Miranda’s theorem, and Zalcman’s lemma in the setting of
total derivatives. Subsequently, in $2019$, they obtained several normality criteria involving bounds on higher order total derivatives (see~\cite{CAO}). More precisely, in \cite[Theorem 1.8, p. 5]{CAO}, they proved
\begin{theorem}\label{1.9}
     Let $k \geq 1$ be a natural number and $\mathcal{F}$ be a family of holomorphic functions in a domain $U \subset \mathbb{C}^{n}$. Suppose that each $f \in \mathcal{F}$  is non-vanishing in $\Delta(0,\delta) \subset U$, where $\delta \in (0,1)$. If for three distinct finite complex numbers $c_1,~c_2,~c_3$ and for all $f \in \mathcal{F}, f(z) = c_i \implies|D_*^{k}f(z)| \leq C,$ for some $C > 0, ~\text{and}~ i = 1,~ 2,~ 3$, then $\mathcal{F}$ is normal in $U$.   
    \end{theorem}
    
    In recent years, the application of differential polynomials has seen a resurgence in the study of normal families of meromorphic functions, since a differential polynomial of a meromorphic function $f$ serves as a natural extension of the derivative of $f$ (see \cite{NIK-1, NIK, CHA, GRA}). Motivated by these recent developments, we consider total differential monomials in $\Cn$ and investigate normality criteria involving such expressions.
\begin{definition}
 Let $f$ be a holomorphic function in a domain $U \subset \Cn$, $m_0,~ m_1,~m_2,\ldots,m_k$ be non-negative integers (not all zeros), and $k\geq1$ be a natural number, then the expression
    \begin{align}\label{1}
    M_{D_*}[f]&:=(D_*f)^{m_1}(D_*^{2}f)^{m_2}(D_*^{3}f)^{m_3}\ldots(D_*^{k}f)^{m_k}
     \end{align}
 is called total differential monomial of $f$. The quantity $\sum\limits_{j=1}^{k}m_j= \bar{d} $ is called the degree of $M_{D_*}[f].$
    \end{definition}

Our second main result shows that the conclusion of Theorem \ref{1.9} remains valid
when the total derivative is replaced by a total differential monomial, thereby
extending Theorem \ref{1.9} to a broader framework as
\begin{theorem} \label{T2}
Let $\mathcal{F}$ be a family of holomorphic functions in a domain $U \subset \mathbb{C}^{n}$, $C > 0$, and $c_1,~c_2,~\ldots,~c_{\bar{d}+2}$ be $\bar{d}+2$ distinct finite complex numbers. Suppose that each $f \in \mathcal{F}$  is non-vanishing in $\Delta(0,\delta) \subset U$, where $\delta \in (0,1)$. If for each $f \in \mathcal{F},~ f(z) = c_i \implies|{M_{D_{*}}}[f](z)| \leq C,~ i = 1,~ 2,~\ldots,~\bar{d}+2$, then $\mathcal{F}$ is normal in $U$.
\end{theorem}

\smallskip

\noindent By taking $M_{D_*}[f] = D_*^{k}(f),$ where $k$ is a natural number, one can easily recover Theorem \ref{1.9}, showing that Theorem \ref{T2} is a generalization of Theorem \ref{1.9}. Furthermore, it can be demonstrated by the following example that the condition ``each $f \in \mathcal{F}$  is non-vanishing in $\Delta(0,\delta)$, where $\delta \in (0,1)$" in Theorem \ref{T2} is essential .
\addtocounter{theorem}{-1}
\begin{example} Let $U =\{ z= (z_1,~z_2,~ z_3,~ z_4) \in \mathbb{C}^{4} :\|z\|<1\} ~\mbox{and}~\F =\{f_t: f_t(z)=t((z_1)^{3} +(z_2)^{3} + (z_3)^{3} + (z_4)^{3}),~~ t \in \mathbb{N}, z \in U\}.$ Then $f_t(0)=0.$ Set $M_{D_*}[f]:=D_*fD_{*}^{k}f,$ where $k \in \mathbb{N}.$ Then, $M_{D_*}[f_t](z)= 3^{k+1}(f_t(z))^{2}.~\mbox{Taking}~ C = 3^{k+1}\max\{|c_i|^2, ~ i=1,~2,~\ldots,~\bar{d}+2\}$, one can see that  $f_t(z) = c_i \implies|M_{D_{*}}[f_t](z)| \leq C,~ i = 1,~2,~\ldots,~\bar{d}+2$. For any $z_0 \in U\setminus\{0\},~ |f_t(z_0)| \longrightarrow \infty$, for sufficiently large $t$. Also, $f_t(0) \longrightarrow 0, t \longrightarrow \infty$, making $\F$ not normal in U. Therefore, the condition ``each $f \in \mathcal{F}$  is non-vanishing in $\Delta(0,\delta) \subset U$, where $\delta \in (0,1)$" in Theorem \ref{T2} cannot be omitted .
\end{example}

\vspace{1em}

\noindent

 Finally, motivated by the differential polynomial studied by Manket and Nevo in \cite{GRA}, we define a non-homogeneous total linear differential polynomial in $\Cn$ as
 \begin{definition}
 Let $f$ be a holomorphic function in a domain $U \subset \Cn$, and $k\geq1$ be a natural number, then the expression
\begin{equation}\label{2}
 \alpha_kD_{*}^{k}f~ +~ \alpha_{k-1}D_{*}^{k-1}f~ +~ \alpha_{k-2}D_{*}^{k-2}f~ +\ldots +~ \alpha_{1}D_*f~ +~ \alpha_{0}f,
\end{equation}
where $\alpha_k, ~\alpha_{k-1}, ~\alpha_{k-2},~ \ldots,~ \alpha_0$ are fixed holomorphic functions in $U$ with $\alpha_k \not\equiv 0,$ is called non-homogeneous total differential polynomial of $f$.
\end{definition}
Manket and Nevo in \cite{GRA} proved the following result in the setting of one complex variable:
\begin{theorem}\label{1.11}
Let $\F$ be a family of non-vanishing holomorphic functions in a domain $ \Omega \subset \C$, $M>0$, and $\alpha_{l-1}(\zeta),~\alpha_{l-2}(\zeta),~ \ldots,~\alpha_{0}(\zeta)$ be fixed holomorphic functions in $\Omega$ with $l\geq 1$ such that \begin{equation*}
| f^{(l)}(\zeta) + \sum\limits_{i=0}^{l-1}\alpha_{i}(\zeta)f^{(i)}(\zeta) | > M ,~~f \in \F,~ \zeta \in \Omega.\end{equation*}
 Then $\F$ is normal in $\Omega$.
\end{theorem}
\begin{question}\label{1.12}
Can we generalize the Theorem \ref{1.11} in the setting of several complex variables $?$
\end{question}
Our third main result, by using total derivative, provides positive answer to the Question \ref{1.12} as
\begin{theorem}\label{T3}
Let $\F$ be a family of non-vanishing holomorphic functions in a domain $ U \subset \Cn$, $C>0$, and $\alpha_{l-1}(z),~\alpha_{l-2}(z),~ \ldots,~\alpha_{0}(z)$ be fixed holomorphic functions with $l\geq 1$ such that \begin{equation}\label{2.3.1}
| D_{*}^{l}f(z) + \sum\limits_{i=0}^{l-1}\alpha_{i}(z)D_{*}^{i}f(z) | > C ,~~f \in \F,~ z \in U.
\end{equation}
 Then $\F$ is normal in $U$.
\end{theorem}
The following example shows that the condition `` each $f \in \F$ is non-zero" is essential.
\addtocounter{theorem}{-1}
\begin{example}
Let $U = \{z \in \C: \frac{7}{10} < |z| < \frac{3}{2}\},$ and $\F =\{f_t: f_t(z) = t(z-1)^{2},~ t \in \mathbb{N}\}.$ be a family of non-vanishing holomorphic functions in $U.$ By easy calculations, one can see that, for all $z \in U,$ $$ |D_{*}^{2}f_t(z)+ D_{*}f_t(z) +f_t(z)| = t|7z^{2} -6z +1| \longrightarrow \infty ~\mbox{as}~ t \longrightarrow \infty.$$ But $\F$ is not normal in $U.$
\end{example}

\section{Preliminary Lemmas}
\noindent
This section is devoted to the pre-existing results that will be used in establishing the proofs of our main results.
\smallskip

The proofs of our results rely heavily on Nevanlinna’s value distribution theory in one complex variable. Therefore, we assume that the reader is familiar with the standard notations of $T(r,f), N(r,f), \overline{N}(r,f), m(r,f), S(r, f), $ and first as well as second fundamental theorem of Nevanlinna. Additionally, for more deeper insights of this theory, the reader is referred to \cite{HAY}.
\medskip

\noindent 
The first lemma, due to Zhu et al.~\cite{JCA}, is the
several-variable analogue of the rescaling result of Tan and Thin for $\varphi$- normal functions in one complex variable \cite[Lemma 1, p. 51]{TAN}. 
\begin{lemma}\label{A1}
Let $f$ be a holomorphic function in $\mathbb{B}\subset \Cn$ and $\varphi: [0,1)\longrightarrow (0,\infty)$ be a smoothly increasing function. Then, $f$ is a $\varphi$- normal function in $\mathbb{B}$ iff corresponding to every sequence $ \{a_t\} \subset \mathbb{B} ;~ \|a_t\| \longrightarrow 1,$ the family
$$ \F = \left\{h_t(w) = f\left(a_t + \frac{1}{\varphi(\|a_t\|)}w\right) ; ~ t\in \mathbb{N}\right\}$$
is normal in $\mathbb{B}.$
\end{lemma}

What follows is an extension of the famous Zalcman Pang lemma to $\mathbb{C}^n$, due to Charak and Kumar~\cite{RAH1}.
\begin{lemma}\label{Z1}
 Let $k$ be a positive integer and $\mathcal{F}$ be a family of holomorphic functions in the unit ball $\mathbb{B}$. Then , $\F$ is normal in $\mathbb{B}$, if for each $p \in [0, 1),$ there does not exist a  
 real number $0<r<1$, a sequence $\{z_t\}\subset\mathbb{B}$ satisfying $0<\|z_t\|<r$, a sequence $\{f_t\}\subset\mathcal{F}$, and a sequence $\{\rho_t\}$ of positive real numbers with $\rho_t\longrightarrow 0$, such that
$$h_t(w) := \frac{f_t\bigl(z_t+ \rho_t w\bigr)}{\rho_j^{p}}$$
converges locally uniformly to a non-constant holomorphic function $h$ in $\Cn$.   
\end{lemma}

In the next lemma we present the several variable analogue of Hurwitz theorem \cite[p. 26]{LEB}.
\begin{lemma}\label{H1}
Let $\mathcal{F}$ be a family of holomorphic functions in domain $U\subset \Cn,$ and $\{f_t\}_{t \in \mathbb{N}} \subset \F$ be a sequence of non-vanishing holomorphic functions. Suppose $\{f_t\}$ converges locally uniformly in $U$ to a holomorphic function $f$ in $U$. Then, either $f$ is also non-vanishing in $U$ or is identically zero in $U.$
\end{lemma}

Finally, we give an extension of the famous Zalcman’s lemma due to Cao and Liu \cite{CAO}, which will serve as a fundamental tool in the proofs of our second and third main results.
\begin{lemma}\label{L1}
Let $\mathcal{F}$ be a family of holomorphic functions in the unit ball $\mathbb{B}$, and suppose that
$f(z) \neq 0$ for every $f \in \mathcal{F}$ and for all $z \in \Delta(0,\delta)$, where $0<\delta<1$. If $\mathcal{F}$ is not normal in $\mathbb{B}$, then for every real number $k$ with $-1<k<1$, there exist a real number $0<r<1$, a sequence $\{z_t\}\subset\mathbb{B}$ satisfying $0<\|z_t\|<r$, a sequence $\{f_t\}\subset\mathcal{F}$, and a sequence $\{\rho_t\}$ of positive real numbers with $\rho_t\longrightarrow 0$, such that
$$F_t(\zeta):= \frac{f_t\bigl(z_t e^{\rho_t \zeta}\bigr)}{\rho_t^{k}},\qquad \zeta\in\mathbb{C}$$
converges locally uniformly in $\mathbb{C}$ to a non-constant entire
function $F$.

Moreover, if $f(z)\neq 0$ for all $f\in\mathcal{F}$ and all
$z\in\mathbb{B}$, then the parameter $k$ can be chosen from
$(-1,\infty)$.

\end{lemma}

\section{Proofs of the main results}

\noindent
\begin{proof}[\bf{Proof of Theorem \ref{T1}}]
On contrary, assume that $f$ is not $\varphi$-normal in $\mathbb{B}$. Then, by Lemma \ref{A1}, the family $\F = \left\{h_t(w) = f\left(a_t + \frac{1}{\varphi(\|a_t\|)}w\right) ; ~ t\in \mathbb{N}\right\}$
is not normal at some $w_0 \in \mathbb{B}.$ In view of Lemma \ref{Z1}, there exists sequences $\{h_{t,i}\}_{i \in \mathbb{N}} \subset \mathcal{F}$ (for ease, we shall denote $\{h_{t,i}\}_{i \in \mathbb{N}}~\mbox{by}~\{h_{t}\} ), \{w_{t}\} \subset \mathbb{B}~\mbox{with}~w_t\longrightarrow w_0,~ \rho_t \subset (0,1)~\mbox{with}~ \rho_t \longrightarrow 0$ such that
\begin{equation}\label{E1}
H_t(w) = h_t(w_t + \rho_tw) =  f\left(a_t + \frac{1}{\varphi(\|a_t\|)}(w_t+\rho_tw)\right)\longrightarrow H(w)
\end{equation}
locally uniformly to a non-constant holomorphic function $H$ in $\Cn.$ Since, $H$ is a holomorphic function, one can easily compute that
\begin{equation}\label{E2}
D^IH_t(w) = \left(\frac{\rho_t}{\varphi(\|a_t\|)}\right)^{|I|}D^If\left(a_t + \frac{1}{\varphi(\|a_t\|)}(w_t+\rho_tw)\right) \longrightarrow D^IH(w)
\end{equation}
locally uniformly in $\Cn$, for all multi-indexes $I \in  \mathbb{N}^n.$
\vspace{1em}

\noindent
\textbf{Claim 1:} All zeros of $H$ are of multiplicity at least $k$.
\vspace{1em}

\noindent
Suppose $T$ be a compact set containing $w_0$ such that $H(w_0)=0$. Then, by Lemma \ref{H1}, there exists sequence $w_t^{'} \longrightarrow w_0 ~\mbox{such that}~ H(w_t^{'}) = 0,$ that implies
$$H_t(w_{t}^{'}) =f\left(a_t + \frac{1}{\varphi(\|a_t\|)}(w_t+\rho_tw_{t}^{'})\right)=0.$$
For convenience, set $\tilde{w_{t}} = a_t + \frac{1}{\varphi(\|a_t\|)}(w_t+\rho_tw_{t}^{'}).$ We find that for sufficiently large $t, \rho_t$ tends to $0$, and so $w_t^{'} \in \mathbb{B}.$ Therefore, in view of (\ref{2.1}) , there exists $C_1 > 0,$ such that
$$|D^If(\tilde{w_t})| \leq C_1,~~~~~~|I| = 1,2,\ldots,k-1. $$
Using the properties of $\varphi,$ we obtain that
\begin{equation}\label{E3}
    D^IH_t(w_{t}^{'}) = \left(\frac{\rho_t}{\varphi(\|a_t\|)}\right)^{|I|}D^If(\tilde{w_{t}}) \leq\left(\frac{\rho_t}{\varphi(0)}\right)^{|I|}D^If(\tilde{w_{t}}).
\end{equation}
Now, (\ref{E2}) and (\ref{E3}) together yield\\
$D^{I}H(w_0)= 0,~~~~~~|I|=1,2,\ldots,k-1,$
and $D^{J}H \not\equiv 0$.\\\\
Therefore, Claim 1 is established.
\vspace{1em}

\noindent
\textbf{Claim 2:} $D^{J}H(z) = 0$ whenever $H(z) \in E.$
\vspace{1em}

\noindent
Again, assume that $T_1$ is a compact subset of $\Cn$ containing $z_0$ such that $H(z_0) = b, ~\mbox{for some}~ b \in E$. Thus, by Lemma \ref{H1}, there exists sequence $\{z_t\}~\mbox{with}~z_t\longrightarrow z_0$ such that $H(z_t) = b,$ which gives
$$H_t(z_{t}) =f\left(a_t + \frac{1}{\varphi(\|a_t\|)}(w_t+\rho_tz_{t})\right)=b.$$
Again, for ease, set $ a_t + \frac{1}{\varphi(\|a_t\|)}(w_t+\rho_tz_{t}) = \tilde{z_t}$. Thus, in view of (\ref{2.2}), and for sufficiently large $t$,  there exists $C_2 > 0$, such that
\begin{equation}\label{E4}
\frac{1}{\varphi(\|\tilde{z}_t\|)^{k}}\frac{|D^{J}f(\tilde{z}_t)|}{1 +|f(\tilde z_t)|^{k+1}} \leq C_2.
\end{equation}
This implies,
$$\frac{|D^{J}H_t(z_t)|}{1 +|H_t(z_t)|^{k+1}}=\left(\frac{\rho_t}{\varphi(\|a_t\|)}\right)^{k}\frac{D^Jf(\tilde z_{t})}{1+|f(\tilde{z_{t}})|^{k+1}}\leq(\rho_t)^k\left(\frac{\varphi(\|\tilde z_t\|)}{\varphi(\|a_t\|)}\right)^{k}C_2$$
 as $t \longrightarrow\infty$. Again,  using (\ref{E2}) , we find that $\frac{|D^{J}H(z_0)|}{1 +|H(z_0)|^{k+1}} =0,$ and thus $D^{J}H(z_0)=0.$\\\\
This establishes Claim 2.
\vspace{1em}

\noindent
\textbf{Claim 3:} $D^{L}H(z) = 0$ whenever $H(z) \in E.$
\vspace{1em}

\noindent
Proceeding as in Claim $2$, and using (\ref{2.2}), we find that for sufficiently large $t$, there exists $C_3 > 0,$ such that
\begin{align}\nonumber\label{E5}
|D^{J}f(\tilde z_t)|& \leq \varphi(\|\tilde z_t\|)^k(1 +|f(\tilde z_t)|^{k+1}) C_3\\
&\leq \varphi(\|\tilde z_t\|)^k(1 +\max_{\beta \in E}|\beta|^{k+1})C_3.
\end{align}
 Again, by (\ref{2.3}), there exists $C_4 > 0$ such 
 that
 \begin{equation}\label{E6}
 \frac{|D^{L}f(\tilde z_t)|}{1 +|D^{J}f(\tilde z_t)|^{(k+1)/k}} \leq C_4.
 \end{equation}
 Set $C = \text{max}\{C_3, C_4\}.$\\\\
 Now, (\ref{E5}) and (\ref{E6}) yield
 \begin{align*}
 \frac{|D^{L}H_t(z_t)|}{1 +|D^{J}H_t(z_t)|^\frac{k+1}{k}} &= \left(\frac{\rho_t}{\varphi(\|a_t\|)}\right)^{k+1}\frac{|D^Lf(\tilde{z_{t}})|}{1+\left(\frac{\rho_t}{\varphi\|a_t\|}\right)^{(k+1)}|D^{J}f(\tilde{z_{t}})|^\frac{k+1}{k}}\\
 &\leq  \left(\frac{\rho_t}{\varphi(\|a_t\|)}\right)^{k+1}\frac{|D^Lf(\tilde{z_{t}})|}{1+|D^{J}f(\tilde z_t)|^\frac{k+1}{k}}\frac{1+|D^{J}f(\tilde z_t)|^\frac{k+1}{k}}{1+\left(\frac{\rho_t}{\varphi\|a_t\|}\right)^{(k+1)}|D^{J}f(\tilde{z_{t}})|^\frac{k+1}{k}}\\
 &\leq  C\left(\frac{\rho_t}{\varphi(\|a_t\|)}\right)^{k+1}\left(1 +[C(1+\max_{\beta \in E}|\beta|^{k+1})]^\frac{k+1}{k}\varphi(\|\tilde z_t\|)^{k+1}\right)\\
 &\leq  C \rho_t^{k+1}\left(1 +[C(1+\max_{\beta \in E}|\beta|^{k+1})]^\frac{k+1}{k}\left(\frac{\varphi(\|\tilde z_t\|)}{\varphi(\|a_t\|)}\right)^{k+1}\right).
 \end{align*}
In view of (\ref{E2}), we obtain that $\frac{|D^{L}H(z_0)|}{1 +|D^{J}H(z_0)|^{k+1}} \longrightarrow 0~ \text{as}~ t \longrightarrow \infty.$ Therefore, $D^{L}H(z_0)=0,$  and hence the Claim 3.
\vspace{1em}

\noindent
Now, for any $ w = (w_{1},~w_{2}~\ldots,~w_{n}) \in \Cn$, define,
$$g_w(\mu) = H(\mu w) = H (\mu w_{1},~\mu w_{2},~\ldots,~\mu w_{n}), ~\mu \in \C.$$
For brevity, we set $\mu w_i = \mu_i.$ Since, $H$ is non-constant holomorphic function in $\Cn,$ implies, $g_w$ is also a non-constant holomorphic function in $\C.$ Assume that $ H(\mu{'} w) = b,~\text{for some}~ b \in E ~\text{and}~ \mu'\in \C ,~ \text{then},~ g_w(\mu^{'}) \in E.$\\
Furthermore, for positive integers, $i_1,~ i_2,~\ldots,~ i_{k+1},$ we have
\begin{align*}
    g_{w}^{(1)}(\mu{'}) =& \sum\limits_{i_1=1}^{n}w_{i_1}\frac{\partial{H}}{\partial{\mu_{i_1}}}(\mu^{'}w),\\
    g_{w}^{(2)}(\mu{'}) =& \sum\limits_{i_1, i_2=1}^{n}w_{i_1}w_{i_2}\frac{\partial^{2}{H}}{\partial{\mu_{i_1}}{\partial\mu_{i_2}}}(\mu^{'}w).
    \end{align*}
    Proceeding in the same way as above, we find that
    \begin{align}\label{E7}
 g_{w}^{(k)}(\mu{'}) &= \sum\limits_{i_1, i_2,\ldots,i_k=1}^{n}w_{i_1}w_{i_2}\ldots w_{i_k}\frac{\partial^{k}{H}}{\partial{\mu_{i_1}}\partial{\mu_{i_2}}\ldots\partial{\mu_{i_k}}}(\mu^{'}w)\\
 g_{w}^{(k+1)}(\mu{'}) &= \sum\limits_{i_1, i_2,\ldots,i_{k+1}=1}^{n}w_{i_1}w_{i_2}\ldots w_{i_{k+1}}\frac{\partial^{k+1}{H}}{\partial{\mu_{i_1}}\partial{\mu_{i_2}}\ldots\partial{\mu_{i_{k+1}}}}(\mu^{'}w)\label{E8}
\end{align}
Now, Claim $2$ and Claim $3$ along with (\ref{E7}) and (\ref{E8}), yield
$$g_{w}^{(k)}(\mu{'}) =0 ~\text{and}~g_{w}^{(k+1)}(\mu{'}) = 0.$$
Now, by using Second fundamental theorem of Nevanlinna, we obtain that
\begin{align*}
    T(r, g_w) &\leq \overline{N}(r, g_w) + \sum\limits_{b_j \in E}\overline{N}\left(r,\frac{1} {(g_w-b_j)}\right) + S(r, g_w)\\
    &\leq \overline{N}\left(r,\frac{1} {g_w^{(k)}}\right)+ S(r, g_w)\\
    &\leq \frac{1}{2}N\left(r,\frac{1} {g_w^{(k)}}\right)+ S(r, g_w)\\
    &\leq  \frac{1}{2}T(r,g_w^{(k)})+ S(r, g_w)\\
    &\leq  \frac{1}{2}T(r,g_w)+ S(r, g_w).
\end{align*}
which is a contradiction.\\
Thus $f$ is $\varphi$-normal. 
\end{proof}

\begin{proof}[\bf{Proof of Theorem \ref{T2}}] 
Assume that $\F$ is not normal in $U$. Then, by Lemma \ref{L1}, there exist a real number $r \in (0, 1), ~\mbox{sequences}~ \{z_{t}\} \subset U : 0 <\|z_{j}\|  < r, \{ f_{t}\} \subset \mathcal{F},~\mbox{and}~ \{\rho_{t}\} \subset (0, 1)~;~  \rho_{t} \to 0 ~\mbox{as}~t \to \infty,~ \mbox{such that}~ $\\
$$h_{t}(\eta ) = f_{t}(z_{t}e^{\rho_{t}\eta}),~~~~~~~~~~~~~~~~~                 \eta \in \C$$\\
converges locally uniformly to a non-constant entire function $h(\eta)$ in $\C.$ Since, $h_t$ is a holomorphic function in $\C$, we can easily compute that 
\begin{equation}\label{4.1}
    h^{(k)}_{t}(\eta ) = \rho_t^{k}D_{*}^{k}f_{t}(z_{t}e^{\rho_{t}\eta}) \longrightarrow h^{(k)}(\eta)
    \end{equation}
    locally uniformly in $\C$.\\
Now, 
\begin{align}\nonumber\label{4}
M_{D_{*}}[f_t](z_{t}e^{\rho_{t}\eta}) &=\prod_{j=1}^{k}\left(D_{*}^{j}f_{t}(z_{t}e^{\rho_{t}\eta})\right)^{m_j}\\
 &= \prod_{j=1}^{k}{\left(\frac{h_{t}^{(j)}(\eta)}{\rho_{t}^{j}}\right)^{m_j}}.
\end{align}
On combining equations (\ref{4.1}) and (\ref{4}), it becomes obvious that
\begin{equation}\label{5}
\rho_{t}^{\sum\limits_{j=1}^{k}{jm_j}}M_{D_{*}}[f_t](z_{t}e^{\rho_{t}\eta}) = \prod_{j=1}^{k}{\left({h_{t}^{(j)}(\eta)}\right)^{m_j}} \longrightarrow \prod_{j=1}^{k}{\left({h^{(j)}(\eta)}\right)^{m_j}}    
\end{equation}
locally uniformly in $\C.$\\\\
\textbf{Claim:} $h(\eta) = c_i \implies \prod_{j=1}^{k}{\left({h^{(j)}(\eta)}\right)^{m_j}} =0,~~~~~i= 1,~ 2,~\ldots~\bar{d} +2$.\\\\
Let $h(\eta_0) = c_1.$ Then by Hurwitz theorem, there exists sequence $\{\eta_t\} \subset \C ~\mbox{with}~ \eta_t \longrightarrow \eta_0 ~\mbox{as}~ t \longrightarrow \infty~\mbox{ such that}~h_t(\eta_t) = f_{t}(z_{t}e^{\rho_{t}\eta_t}) = c_1$, it implies that $|M_{D_{*}}[f](z_{t}e^{\rho_{t}\eta_t})| \leq C.$ Using the assumptions of the theorem and the equation (\ref{5}), we obtain
\begin{equation*}
 \left|\prod_{j=1}^{k}{\left({h^{(j)}(\eta_0)}\right)^{m_j}}  \right|=\lim_{t\to\infty}\left|\rho_{t}^{\sum\limits_{j=1}^{k}{jm_j}}M_{D_{*}}[f_t](z_{t}e^{\rho_{t}\eta_t})\right|\leq \lim_{t\to\infty}\rho_{t}^{\sum\limits_{j=1}^{k}{jm_j}}C =0,   
\end{equation*}
which gives $\prod_{j=1}^{k}{\left({h^{(j)}(\eta_0)}\right)^{m_j}} = 0.$ By applying similar reasoning for $c_2, c_3,\ldots, c_{\bar{d}+2},$ we establish our claim.\\\\
In view of above established claim and the Second Fundamental Theorem of Nevanlinna, we find that
\begin{align*}
(\bar{d}+1)T(r, h) &\leq  \overline{N}(r, h) +\sum_{i=1}^{\bar{d}+2}\overline{N}\left(r,\frac{1}{h-c_i}\right) + S(r,h)\\
&\leq N\left(r,\frac{1}{\prod\limits_{j=1}^{k}(h^{(j)})^{m_j}}\right) + S(r, h)\\
&\leq T\left(r,\frac{1}{\prod\limits_{j=1}^{k}(h^{(j)})^{m_j}}\right) + S(r, h)\\
&\leq\sum\limits_{j=1}^{k}m_j T\left(r,\frac{1}{h^{(j)}}\right) + S(r, h)\\
&\leq\sum\limits_{j=1}^{k}m_jT(r,h^{(j)}) + S(r, h)\\
&\leq\sum\limits_{j=1}^{k}m_jT(r,h) + S(r, h)\\
&= \bar{d}~T(r,h) + S(r, h).\\
\end{align*}
Thus, $T(r, h) = S(r,h)$, which is a contradiction.\\ Hence $\F$ is normal in $U.$  
\end{proof}
\noindent

\begin{proof}[\bf{Proof of Theorem \ref{T3}}] Assume that $\F$ is not normal in $U$. Then, by Lemma \ref{L1}, there exist a real number $r \in (0, 1), ~\mbox{sequences}~ \{z_{t}\} \subset U : 0 <\|z_{j}\|  < r, \{ f_{t}\} \subset \mathcal{F},~\mbox{and}~ \{\rho_{t}\} \subset (0, 1)~;~  \rho_{t} \longrightarrow 0 ~\mbox{as}~t \longrightarrow \infty,~ \mbox{such that}~\mbox{for}~ \beta >l$,
$$h_{t}(\eta ) = \frac{f_{t}(z_{t}e^{\rho_{t}\eta})}{\rho_{t}^{\beta}},~\eta \in \C$$\\
converges locally uniformly to a non-constant entire function $h(\eta)$ in $\C.$ Since, $h_t$ is a holomorphic function in $\C$, we can easily compute that 
\begin{equation}\label{3}
    h^{(l)}_{t}(\eta ) = \rho_t^{l-\beta}D_{*}^{l}f_{t}(z_{t}e^{\rho_{t}\eta}) \longrightarrow h^{(l)}(\eta)
    \end{equation}
    locally uniformly in $\C$.\\

We may assume that the sequence $\{z_t\} \longrightarrow z_0$, for sufficiently large $t$.
    Now, for each $ i= 1,~ 2,~ \ldots,~ l-1,$ and $\eta_0 \in \C$, we can see that
    \begin{equation}\nonumber
        \alpha_i(z_{t}e^{\rho_{t}\eta_0})h_{t}^{(i)}(\eta_0) \longrightarrow\alpha_i(z_{0})h^{(i)}(\eta_0),~\mbox{as}~t \longrightarrow \infty.
\end{equation}
Therefore,
\begin{equation}\label{8}
     \left|h_{t}^{(l)}(\eta_0) +\sum_{i=1}^{l-1}\rho_{t}^{l-i}\alpha_i(z_{t}e^{\rho_{t}\eta_0})h_{t}^{(i)}(\eta_0)\right| \longrightarrow |h^{(l)}(\eta_0)|< M,
\end{equation}
for sufficiently large $t \in \mathbb{N}$, and some positive real $M$.\\
Now, \eqref{3} and (\ref{2.3.1}) together yield
 \begin{align}\nonumber\label{9}
     &\left|h_{t}^{(l)}(\eta_0) +\sum_{i=1}^{l-1}\rho_{t}^{l-i}\alpha_i(z_{t}e^{\rho_{t}\eta_0})h_{t}^{(i)}(\eta_0)\right|\\
     =&\left|\rho_{t}^{l-\beta}D_{*}^{l}f_{t}(z_te^{\rho_t\eta_0}) +\sum_{i=1}^{l-1}\rho_{t}^{l-\beta}\alpha_i(z_{t}e^{\rho_{t}\eta_0})D_{*}^{i}f_{t}(z_te^{\rho_t\eta_0})\right|> \rho_{t}^{l-\beta}C\longrightarrow \infty,~~~t \longrightarrow \infty.
 \end{align}
which is a contradiction to (\ref{8}).\\
Hence $\F$ is normal in $U.$  
\end{proof}
\vspace{3em}

\noindent
\textbf{Acknowledgements:} The author is thankful to her supervisor Prof. K. S. Charak for his valuable remarks and constructing criticism during the preparation of this paper. Also, the author is grateful to the Department of Science and Technology, Ministry of Science and Technology, Government Of India for the financial support through DST-INSPIRE Fellowship (No. DST/INSPIRE/03/2022/005759, IF 220689).

\medskip

\noindent{\bf Data Availability:} Data sharing is not applicable since the article is purely theoretical in nature.


\begin{thebibliography}{99}
\bibitem{AL} A. Aleman and A. Constantin, \emph{Harmonic maps and ideal fluid flows}, Arch. Ration. Mech. Anal., \textbf{204} (2012), 479-513.
\bibitem{ARB} H. Arbel\'aez, R. Hern\'andez and W. Sierra, \emph{Normal harmonic mappings}, Monatsh. Math., \textbf{190} (2019), 425-439.
\bibitem{AUL} R. Aulaskari and J. R\"atty\"a, \emph{Properties of meromorphic $\varphi$- normal functions}, Michigan Math. J., \textbf{60} (2011), 93-11.
\bibitem{NIK-1} N. Bharti and K. S. Charak, \emph{Value distribution of certain differential polynomials and a counterexample to the converse of Bloch's principle}, Bull. Korean Math. Soc., \textbf{62}(1) (2025), 275-297.
\bibitem{NIK} N. Bharti and N. V. Thin, \emph{Results on normal harmonic and $\varphi$-normal harmonic mappings}, Monatsh. Math., \textbf{207} (2025), 385-403.
\bibitem{BOH} N. Bohra, G. Datt and R. Pal, \emph{Lappan’s five-point theorem for $\varphi$-normal harmonic mappings}, Monatsh. Math., \textbf{206}(4) (2025), 797-808.
\bibitem{LIU} T. B. Cao and Z. X. Liu, \emph{Normality criteria for a family of holomorphic functions
concerning the total derivative in several complex variables}, J. Korean Math. Soc., \textbf{53} (2016), 1391-1409.
\bibitem{RAH1} K. S. Charak and R. Kumar, \emph{Lappan's normality criteria in $\Cn$}, Rend. Circ. Mat. Palermo. (2), \textbf{72} (2023), 239-252. 
\bibitem{CHA} K. S. Charak, N. Bharti and A. Singh, {\em Normality through partial sharing of sets with differential polynomials}, Filomat, \textbf{39}(24) (2025), 8273-8288.
\bibitem{DENG} H. Deng, S. Ponnusamy and J. Qiao, \emph{Properties of normal harmonic mappings}, Monatsh. Math., \textbf{193} (2020), 605-621.
\bibitem{DOV} P. V. Dovbush, \emph{Zalcman's lemma in $\Cn$}, Complex Var. Elliptic Equ., \textbf{65} (2020), 796-800 .

\bibitem{HAY} W. K. Hayman, Meromorphic Functions, Clarendon Press, Oxford, (1964).
\bibitem{RAH3} R. Kumar, \emph{Zalcman Pang lemma concerning the total derivatives in several complex variables}, J. Anal., \textbf{33} (2025), 2547-2557.  
\bibitem{LEB} J. Lebl, \emph{Tasty Bits Of Several Complex Variables}, https://www.jirka.org/scv/, (2023).
\bibitem{LEH} O. Lehto and K. I. Virtanen, \emph{Boundary behaviour and normal meromorphic functions}, Acta Math., \textbf{97} (1957), 47-65 .
\bibitem{CAO} Z. Liu and T. Cao, \emph{Zalcman's lemma and normality concerning shared values of holomorphic functions and their total derivatives in several complex variables}, Rocky Mt. J. Math., \textbf{49}(8) (2019), 2689-2716 .
\bibitem{JIN} J. Lu, \emph{Theorems of Picard type for entire functions of several complex variables}, Kodai Math. J., \textbf{26} (2003), 221-229.
\bibitem{GRA} T. Manket and S. Nevo, \emph{On row differential inequalities related to normality and quasi-normality}, Compt. Method Funct. Theory, \textbf{24} (2024), 479-511. 
\bibitem{MON} P. Montel, \emph{Sur les families de fonctions analytiques qui admettent des valeurs exceptionnelles dans un domaine}, Ann. \'Ecole. Norm. Sup., \textbf{29}(3) (1912), 487-535.
\bibitem{NOS} K. Noshiro, \emph{Contributions to the theory of meromorphic functions in the unit circle}, J. Fac. Sci. Hokkaido Univ., \textbf{7} (1938), 149-159.
\bibitem{PAN} S. Pandita and N. Bharti, {\em Bicomplex analogues of some normality criteria}, J. Anal., \textbf{32} (2024), 3515-3531.
\bibitem{SCH} J. L. Schiff, \emph{Normal Families}, Springer, Berlin, (1993).
\bibitem{TAN} T. V. Tan and N. V. Thin, \emph{On Lappan’s five-point theorem}, Comput. Methods Funct. Theory, \textbf{17} (2017), 47-63 .
\bibitem{YOS} K. Yosida, \emph{On a class of meromorphic functions}, Proc. Phys. Math. Soc. Japan, (1934), 227-235 .
\bibitem{JCA} T. Zhu, S. Zhou and L. Yang, \emph{On normal functions in several complex variables}, J. Classical Anal., \textbf{16}(1) (2020), 45-58.
\end{thebibliography}
\end{document}